\documentclass[12pt]{amsart}
\usepackage[active]{srcltx}
\usepackage[all]{xy} \xyoption{poly}
\usepackage{xcolor}
\usepackage{xspace}
\usepackage{enumitem}
\usepackage[mathscr]{eucal}
\usepackage{amssymb}
\usepackage{a4wide}
\usepackage{ifthen}
\usepackage{amscd}
\usepackage[hyperindex,bookmarksnumbered,
plainpages,backref]{hyperref}

\def\kk{\mathbf k}
\newcommand{\g}{\mathfrak{g}}

\newcommand{\rr}{\mathfrak{r}}

\newcommand{\so}{\mathfrak{so}}
\newcommand{\ssp}{\mathfrak{sp}}
\newcommand{\ssl}{\mathfrak{sl}}

\def\Z{\mathbb Z}
\def \O{\mathbb O}
\def\half{\frac12}
\def\End{\mathrm{End}}
\def\Rad{\mathrm{Rad}}

\newcommand{\R}{\mathcal{R}}

\newcommand{\cJ}{J}
\newcommand{\I}{\mathcal{I}}

\newcommand{\Hom}{\operatorname{Hom}}
\newcommand{\rad}{\rm {rad}}
\newcommand{\cS}{\mathcal{S}}

\newcommand{\HH}{\ensuremath{\mathfrak{h}}}
\newcommand{\Jhalf}{\cJ{\rm-mod}_{\frac 12}}
\newcommand{\Jone}{\cJ{\rm-mod}_1}
\newcommand{\Jzero}{\cJ{\rm-mod}_0}
\newcommand{\Jmod}{\cJ{\rm-mod}}

\newtheorem{thm}{Theorem}[section]
\newtheorem{prop}[thm]{Proposition}
\newtheorem{lem}[thm]{Lemma}
\newtheorem{Cor}[thm]{Corollary}
\newtheorem{Rem}[thm]{Remark}

\newtheorem{ex}[thm]{Example}

\makeindex

\begin{document}

\title[Special modules over Jordan algebras]
{Special modules over Jordan algebras}

\author{Iryna Kashuba}
\address{Shenzhen International Center for Mathematics,
Southern University of Science and Technology, China}
          \email{kashuba@sustech.edu.cn}
\author{Vera Serganova}
\address{Department of Mathematics, University of California at Berkeley, Berkeley, CA 94720, USA}
          \email{serganov@math.berkeley.edu}
\begin{abstract}

In this paper we study special representations of finite-dimensional Jordan algebra $J$ 
whose $\Rad^2 J=0$. For each Jordan algebra $J$ of this class we consider its 
Tits-Kantor-Koecher construction $TKK(J)$ and then associate to the latter a quiver 
with relations $Q$ such that the category of representations 
of $Q$ is isomorphic to the category of special representations of $J$. 

\end{abstract}

\maketitle

\section{Introduction}

Jordan algebras arose in early 1930s from the search for an “exceptional” setting for quantum mechanics to replace  the Copenhagen model, where the physical observables are represented by infinite Hermitian matrices. Extracting their formal algebraic properties without reference to the underlying matrix algebra, Pascual Jordan  considered a commutative algebra $J$ which satisfies some partial associativity identity
\begin{equation}\label{jor_alg}
\begin{array}{c}
((a\circ a)\circ b)\circ a= (a\circ a)\circ (b\circ a), \quad a,b\in J.
\end{array}
\end{equation}
Such algebras now carry his name. Due to its origin one can obtain a Jordan algebra $A^+$ from  any associative algebra $A$ substituting the original product by symmetric product
\begin{equation}\label{special}
a\circ b=ab+ba \quad a,b\in A.
\end{equation} A Jordan algebra is called {\it special} if it can be embedded into some $A^+$, otherwise it is called 
{\it exceptional}. 

Special and exceptional algebras are quite different in nature thus the speciality problem became the cornerstone question in Jordan theory. It is reflected in the representation theory as well.
The formal definition of a (bi)module over a Jordan algebra $J$ due to Eilenberg says that it is a vector space 
$M$ endowed with a mapping $a\to \rho(a)$ of  $J$ to the associative algebra $ \operatorname{End}(M)$ such that the split extension algebra $J\oplus M$ is Jordan algebra as well. Then $M$ or equivalently  (bi)representation $\rho$ is characterized by a list of quite hefty identities \eqref{representation_relation} of degree three deduced from \eqref{jor_alg}. Denote by $\Jmod$ a category of Jordan modules over $J$.

Another natural source of mappings of Jordan algebra $J$ into associative algebra $\operatorname{End}(M)$, suggested by \eqref{special}, are Jordan homomorphism of $J$ into $\operatorname{End}(M)^+$.  
Interestingly for a Lie algebra $\g$ these two types of mappings coincide, since representations of $\g$ are given by Lie homomorphisms of $\g$ into 
$\operatorname{End}(M)^-$. The reason for this is that any Lie algebra $\g$ is special, as it can be embedded into $A^-$ for some associative algebra $A$. For Jordan algebras it is straightforward to check that any Jordan homomorphism $\sigma: J\to \operatorname{End}(M)^+$ satisfies  
\eqref{representation_relation}, thus $M$ has the structure of a $J$-module. We will call such module $M$ and representation $\sigma$ {\it special} or {\it one-sided}. 

For unital Jordan algebras one has another characterisation of special modules. Let $e$ be the identity element of $J$ and 
$\rho$ be a representation of $J$ corresponding to Jordan module $M$. Then 
relation \eqref{representation_relation} when all entries are equal to $e$,
$\rho(e)(\rho(e)-1)(2\rho(e)-1)=0$, implies existence of the Peirce decomposition 
\begin{equation}\label{decomposition-module}
M=M_0\oplus M_{\frac12}\oplus M_1,
\end{equation}
where $M_i=\{m_i\,|\, m_i\rho(e)=im_i\}$, $i=0,\frac12,1$.  Denote by $\Jhalf$ (respectively by $\Jone$) a category of modules over $J$ on which $e$ acts as $\frac12$ (respectively as $1$) and by $\Jzero$ a category of trivial modules over $J$. Then \eqref{decomposition-module} rewrites
as 
$$
\Jmod=\Jzero\oplus\Jhalf\oplus\Jone.
$$
Let $\rho(e)=0$ then applying \eqref{representation_relation}  for $b=c=e$ one obtains $\rho(a)=0$ for any $a\in J$,
thus $M_0$ is the trivial module. Let $\rho(e)=\frac12$ then \eqref{representation_relation} for $c=e$ implies that $\rho(a\circ b)=\rho(a)\rho(b)+\rho(b)\rho(a)$, therefore $\rho: J\to \operatorname{End}(M_{\frac12})$ is a special representation of $J$ in $M_{\frac12}$. 
More precisely $\Jhalf\oplus \Jzero$ is the category of all special modules over $J$. For each of these categories one can introduce the universal objects, associative algebras $U(J)$ and $U_i(J)$ with $U(J){\rm-mod}\simeq \cJ{\rm-mod}$ and $U_i(J){\rm-mod}\simeq \cJ{\rm-mod}_i$,
 $i=0,\frac12, 1$, see next Section for more details.

In \cite{Jac2}  N.~Jacobson proved that if $J$ is a finite dimensional 
semisimple Jordan algebra then both $\Jhalf$ and $\Jone$ are semisimple.  Moreover, if $J$ is a finite dimensional Jordan algebra then $U(J)$ is finite-dimensional as well.  These observations suggest that the representation theory of finite-dimensional Jordan algebras goes in parallel with the representation theory of finite-dimensional associative algebras, in particular, to study representations of non-semisimple algebras one can determine their representation type. Due to \cite{dr} all finite dimensional associative algebras can be
divided into the following three classes: algebras of {\it finite
type} which have only finite number of indecomposable
finite-dimensional modules,  algebras of {\it tame type} for which there  is a finite number of one-parameter families
of indecomposable modules in every dimension, and
algebras of {\it wild type}. A powerful method for determining the representation type of a certain abelian category  is to reduce the question to the category of representations of some pointed algebra, or equivalently, to a quiver with relations.

In \cite{KOSh} authors suggested to study representation type of Jordan algebra $J$ with $\operatorname{Rad}^2 J=0$, where $\operatorname{Rad} J$ is the Jacobson radical of $J$, an algebra which is in some sense the closest to being a semi-simple algebra. The analogous class of associative algebras was a starting point in classifying associative algebras of finite representation type.  A remarkable theorem of Gabriel provides a beautiful combinatorial characterization through the renown Dynkin diagrams  for algebras of finite type and the extended Dynkin diagrams for algebras of tame type, \cite{Ga2}.

It turns out that very few Jordan algebras (even with $\operatorname{Rad}^2 J=0$) have tame unital representation type. In \cite{KS} we show that there are no finite type quivers and only six tame quivers for Jordan algebras of Clifford type, and all  $J$ of Clifford type with  $\operatorname{Rad}^2 J\neq 0$
are wild. Moreover, we believe that there are only finitely many tame quivers for $U_1(J)$ in general.

For special representations of Jordan algebras the situation is quite the opposite. In \cite{KOSh} the authors considered the case when $J$ is a Jordan algebra of matrix type with $\operatorname{Rad}^2J=0$. It was shown that $\operatorname{Rad}^2 U_{\frac{1}{2}}(J)= 0$. Hence one can apply the Dynkin diagram criterium. In the general case, for example for Jordan algebras of Clifford type, $\operatorname{Rad}^2 J=0$ does not imply
$\operatorname{Rad}^2 U_{\frac{1}{2}}(J)\neq 0$. It appears that the algebras $U_{\frac12}(J)$ form a new beautiful class of associative algebras which is interesting beyond the representation theory of Jordan algebras.
In this paper we describe quivers with relations for $U_{\frac{1}{2}}(J)$ of any finite-dimensional Jordan algebra $J$ with $\operatorname{Rad}^2 \cJ=0$ and show that $U_{\frac{1}{2}}(J)$ is Koszul. In the forthcoming paper \cite{BKS} we  show that this class of algebras is a generalization of zigzag algebras and determine when $J$ has finite or tame representation type.

Our main tool is the Tits-Kantor-Koecher construction which allows us to reduce the problem to certain representations of Lie algebras. 
We developed this approach in~\cite{KS2} and review these results in Section~\ref{section-tits-kantor-koecher}.

\section{Jordan algebras and their
representations}\label{section-Jordan-algebra-representation}
We require the field $\,{\bf k}\,$ in addition to being
algebraically closed be of characteristic zero.
In this section we give some basics about finite-dimensional Jordan algebras and their representations,
finishing it with the description of Jordan algebras with $\Rad^2 J=0$.

Two classical examples of special Jordan algebras are the
following. If $(A,j)$ is an associative algebra with involution
$j$ then the set of symmetric elements $H(A,j)$ is a Jordan subalgebra of $A^+$. 
If $V$ is
a $\kk$-vector space $V$ of dimension greater then one equipped
with bilinear form $f(x,y)$, then $\cJ(V,f)=\kk\cdot 1+V$, where
$1$ is a unit and $v\circ u=f(x,y)\cdot1$ for $v,u\in V$ defines a {\it Jordan algebra of a bilinear form}. 
Denote by $C(V,f)$ the
Clifford algebra of $V$ relative to $f$, then
$\cJ(V,f)$ is special since it is a subalgebra of $C(V,f)^+$. 

\begin{thm}\label{albert-theorem}[\cite{jj}, VII.6.13, V.6.2] 
Any finite dimensional Jordan algebra $\cJ$ contains a
unique maximal nilpotent ideal $\R=\Rad \cJ$, which is called the radical of
$\cJ$. A semi-simple algebra $\cJ/\R$  is a direct some of simple
ideals $\I$, where $\I$ is
\begin{enumerate} \item basic field $\kk$,  \item
$\cJ(V,f)$ a Jordan algebra of non-degenerate bilinear form $f$
over $\kk$-vector space $V$,  \item {\it Hermitian algebra}
$H_n(C_i)=H(M_n(C_i,*),j)$, $n\geq 3$, where $(C_i,*)$ is a composition
algebra of dimension $i$, $i=1,2,4$, over $\kk$, and $j$ is an involution of $M_n(C)$
 $j(\{a_{ij}\})=\{({a_{ji}})^*\}$,  \item the
Albert algebra $Alb=H_3(\O)$, where $\O$ is $8$-dimensional
octonion algebra over $\kk$ and $j$ is involution introduced in (3).
\end{enumerate}
\end{thm}
Clearly in Theorem \ref{albert-theorem} algebras $(1)-(3)$ are special, as for $Alb$, in 
1934 Albert showed that it is exceptional.

Let $M$ be a $\kk$-vector space endowed with a pair of linear
mappings $\cJ\otimes_{\kk} M\to M$, $(a,m)\to a\cdot m$ and
$M\otimes_{\kk}\cJ\to M$, $(m,a)\to m\cdot a$, define an algebra
$\Omega=\cJ\oplus M$ with $\kk$-linear product
\begin{equation}
(a_1+m_1)\star(a_2+m_2)=a_1\circ a_2+a_1\cdot m_2+m_1\cdot a_2,
\end{equation}
where $a_1,a_2\in \cJ$, $m_1,m_2\in M$. Then $M$ is called a {\it
Jordan (bi)module} over $\cJ$ if $(\Omega,\star)$ is Jordan algebra.
Equivalently one can consider a mapping $\rho\,:\cJ\to
\End_{\kk}M$, $\rho(a)m=a\cdot m$, $a\in\cJ$, $m\in M$, which
satisfies the following identities $[\rho(a),\rho(a\circ a)]=0$ and 
\begin{equation}\label{representation_relation}
\rho(a)\rho(b)\rho(c)+\rho(c)\rho(b)\rho(a)+\rho((a\circ c)\circ b)=\rho(a)\rho(b\circ
c)+\rho(b)\rho(c\circ a)+\rho(c)\rho(a\circ b),
\end{equation}
$a,b,c\in J$. Such mapping is called a {\it (bi)representation} of $\cJ$. Analogously to the Lie algebra case we define the universal
multiplication envelope $U(J)$ as quotient of tensor algebra
$T\langle \rho(a), a\in\cJ\rangle$ by the ideal $I$ generated by
identities \eqref{representation_relation}. Since $\rho$ endows $M$ with the structure of $J$-module, the category of left
$U(J)$-modules $U(J)-{\rm mod}$ is isomorphic to the category $\Jmod$.

Recall that we call a $J$-module $M$ special (or one-sided) if $(a\circ b)\cdot m=a\cdot(b\cdot m)+b\cdot (a\cdot m)$, $a,b\in J$, $m\in M$. Then
the special universal envelope is a rehash of \eqref{special}:
$$
S(J)=T\langle\cJ\rangle/\langle a\otimes b+b\otimes a-a\circ b,\ a,b\in\cJ\rangle.
$$
There are two homomorphisms relating $U(J)$ and $S(J)$, the epimorphism coming from the universal property of $U(J)$:
$U(J)\twoheadrightarrow S(J)$, consequently $S(J)$ is a quotient of $U(J)$.  The other one $S(J)\otimes S(J)\to U(J)$ originates from a tensor product of two special modules. Let $\sigma_1: J\to \operatorname{End}(M)^+$ and $\sigma_2: J\to  \operatorname{End}(N)^+$ be two special representations  
of $J$ then 
\begin{equation}\label{tensor-product}
\rho(a)=\sigma_1(a)\otimes 1+1\otimes\sigma_2(a), \quad a\in J
\end{equation} defines a representation of $J$ in $ \operatorname{End}(M\otimes N)$ and endows  $M\otimes N$ with a structure of $J$-module
which is called a {\it tensor product of two special modules}.  

Let $J$ be unital algebra, $e$ be the identity element of $J$ and $\rho$ be a representation of $\cJ$ corresponding to Jordan module $M$. 
Then in $U(J)$ one has $2e^3+e=3e^2$, what leads to the existence of three orthogonal idempotents $E_0=(e-1)(2e-1)$,
$E_1=e(2e-1)$ and $E_{\frac12}=-4e(e-1)$ in $U(J)$ such that $E_0+E_1+E_{\frac12}=1$. Which in turn implies the Peirce decomposition
$$
U(J)=U_0(J)\oplus U_{\frac12}(J)\oplus U_1(J),
$$
where $U_i(J)=U(J)E_i,\ i=0,\frac12,1$, is an ideal in $U(J)$. Moreover denote by $a_i$
the component of $a\in J\hookrightarrow U(J)$ in $U_i(J)$ then $a\to a_0+a_{\frac12}$
coincides with a special universal envelope for $J$, equivalently $U_0(J)\oplus U_{\frac12}\simeq S(J)$, see Theorem II.11.15, \cite{jj}.

As we stated in the introduction we want to explore special representations of
finite-dimensional Jordan algebra $\cJ$. Without loss of generality 
we may assume that $J$ has identity element. If $J$ does not contain identity then let $J^{\#}=\cJ\oplus \kk 1$ be the
algebra obtained by the formal adjoining of an identity element to
$\cJ$. Then the category of $S(J)$-mod is isomorphic
to the category $U_{\frac12}(J^\#)$-mod, for proof see Theorem 2.3 in
\cite{KOSh}. By Theorem \ref{albert-theorem} any unital algebra $J$  with $\Rad^2\cJ=0$ can be written
as $\I(J)\ltimes \Rad J$, where $\I(J)$ is a semi-simple Jordan algebra and  $\Rad J$ is 
a unital $\I(J)$-bimodule. By Theorem 19.3, \cite{Jac2}  $\Rad J$ is completely reducible. 
The following lemma finishes the descriptions of $\I(J)\ltimes \Rad J$.
\begin{lem}{\cite{Jac2}}\label{simple_module_over_semi_simple_algebra}
Let $\I=\I_1\oplus\I_2\oplus\dots\oplus\I_r$ be a semi-simple
Jordan algebra and $M$ be irreducible unital $\I$-module then we have one
of the following cases
\begin{enumerate}
\item $M$ is a unital $\I_i$-module for some $1\leq i\leq r$,
\item $M$ is a unital $\I_i\oplus \I_j$ module which is a tensor product of two
special modules over $\I_i$ and $\I_j$ respectively, $1\leq i\neq j\leq r$. 
\end{enumerate}
\end{lem}

In \cite{Jac2} N. Jacobson described both $U_{\frac12}(\I)$ and $U_1(\I)$ 
for any finite-dimensional simple Jordan algebra $\I$ and thus obtained 
full description of $\I$-mod$_{\frac12}$ and $\I$-mod$_1$. In the next section we 
will provide the list of irreducible special unital modules for corresponding Tits-Kantor-Koecher
construction $TKK(\I)$. For Jordan modules over
Hermitian algebras (including Albert algebra) see Theorem 19, \cite{Jac2} and also detailed exposition 
in subsection 4.3.1, \cite{KOSh}, for Jordan algebras of bilinear form see 
Theorem 14.1 for special modules and Theorem 16.1 and 16.2 for unital case, all in \cite{Jac2}.
For the sake of completeness, the list of non-trivial irreducible modules  over field $\kk$ is exhausted by 
one-dimensional special module and 
one-dimensional until module.

\section{The Tits-Kantor-Koecher construction for Jordan algebras and
their  special and unital representations}\label{section-tits-kantor-koecher}
The Tits-Kantor-Koecher construction was introduced independently in \cite{Tits},
\cite{Kantor} and \cite{Koecher}. It associate to any unital Jordan algebra $\cJ$
a Lie algebra $Lie(\cJ)$ with a short grading defined by ${\mathfrak{sl}}_2$ 
subalgebra. A {\em short grading} of an algebra $\g$ is a $\Z$-grading of the
form $\g=\g_{-1}\oplus\g_0\oplus\g_1$. Let $P$ be the commutative
bilinear map on $\cJ$ defined by $\,P(x,y)=x\circ y$. Then we
associate to $\cJ$ a Lie algebra with short grading
$Lie(\cJ)=\g_{-1}\oplus\g_0\oplus\g_1$. Namely
we put $\g_{-1}=\cJ$, $\g_0=\langle
L_a,[L_a,L_b]\,|\,a,b\in\cJ\rangle\subset\End(\g_{-1})$ and $\g_1=\langle
P,[L_a,P]\,|\,a\in\cJ\rangle$ with the following multiplication
\begin{itemize}
\item $[x,y]=0$ for $x,\,y\in \g_{-1}$ or $x,\,y\in \g_1$; \item
$[L,x]=L(x)$ for $x\in\g_{-1}$, $L\in\g_0$; \item
$[B,x](y)=B(x,y)$ for $B\in\g_1$ and $x,y\in\g_{-1}$; \item
$[L,B](x,y)=L(B(x,y))-B(L(x),y)+B(x,L(y))$ for any $B\in\g_1$,
$L\in\g_0$ and $x,y\in\g_{-1}.$
\end{itemize}
Then $\g=Lie(\cJ)$ is a Lie algebra which is called the Tits-Kantor-Koecher (TKK)
construction for $\cJ$, see \cite{Kac}.

A {\em short subalgebra} of a Lie algebra $\g$ is an $\mathfrak{sl}_2$
subalgebra spanned by elements $e,h,f$ such that the eigenspace
decomposition of $ad\,h$ defines a short grading on $\g$. For any
Jordan algebra $\cJ$ with unit element $e$ consider elements of
$Lie(\cJ)$ $h_{\cJ}=-L_e$ and $f_{\cJ}=P$, then $\alpha_{\cJ}$
spanned by $e,h_{\cJ}$ and $f_{\cJ}$ defines short subalgebra of
$Lie(\cJ)$.  If  $\g$ is a Lie algebra and $\alpha$
its short subalgebra we will call a pair $(\g,\alpha)$ a short pair.
For any short pair  $(\g,\alpha)$ and any $x,y\in\g_{-1}$ set 
\begin{equation}\label{jordan_product_tkk}
x\circ y=\left[[f_{\alpha},x],y\right]
\end{equation}
then $Jor(\g):=(\g_{-1},\circ)$ is a Jordan algebra.

\begin{lem}\label{center_of_lie}
Let $(\HH,\alpha)$ be a short pair such that $[\HH_{-1},\HH_1]=\HH_0$. Then there 
is an epimorphism $\theta\,: \HH\to Lie(Jor(\HH))$.
\end{lem}
\begin{proof}
Let $\cJ=Jor(\HH)$ and $(\g,\beta)=Lie(\cJ)$. We set  $\theta:\,\HH_{-1}\simeq \cJ\simeq \g_{-1}$ to be the 
identity map. The condition $[\HH_{-1},\HH_1]=\HH_0$ and $\HH_0=[f,\HH_{-1}]+{\HH_0}^f$ implies that
$$\HH_0=[f,\HH_{-1}]+\left[ [f,\HH_{-1}],[f,\HH_{-1}]\right].$$ 
Now we set for any $x\in\HH_{-1}$ $\theta([x,f])=L_x$, and extend $\theta$ to the homomorphism $\HH_0\to\g_0$.
Finally to define $\theta:\,\HH_{1}\to\g_1$ we use $\theta(f)=P$ and $\HH_1=\left[f,[f,\HH_{-1}]\right]$.
We leave to the reader to check that $\theta$ defines an homomorphism of Lie algebras, see Lemma 5.15, \cite{Kac}.

It is clear from the construction that ${\rm Ker}\theta$
consists of all $x\in\HH_0+\HH_1$ such that $[x,\HH_{-1}]=0$. 
If $x\in\HH_1$ then $[e,x]=[f,x]=0$, hence $[h,x]=0$, thus $x=0$. 
If $x\in\HH_0$ then $[e,x]=0$ implies $[f,x]=0$ and 
we have $\left[x,[f,\HH_{-1}]\right]=0$ and 
$\left[x,[f,[f,\HH_{-1}]]\right]=0$, i.e. $x\in Z(\HH)$. 
\end{proof}

A short pair
$(\g,\alpha)$ is called {\em minimal} if any
 if
$[\g_{-1},\g_{1}]=\g_0$ and $Z(\g)=0$. It is easy to see that for any Jordan algebra $J$, $Lie(J)$ is always minimal.

\begin{thm} [Theorem 5.15, \cite{Kac}]\label{Lie-Jor-alg} The functors $Lie$ and $Jor$ establish the equivalence 
of the categories of unital Jordan algebras and 
of minimal pairs $(\g,\alpha)$, with  morphisms in both categories being
epimorphism.
\end{thm}

It is possible to extend these functors to the categories of representations.  
Denote by $\hat{\g}$ the universal central extension of $\g$. Observe that $\hat{\g}$
has a short grading, the center of $\hat{\g}$ is in $\hat{\g}_0$, and $\hat{\g}_{\pm 1}=\g_{\pm 1}$.
Denote by $\mathcal{S}_{\frac12}$  the category of finite-dimensional $\hat{\g}$-modules $V$ such that as $\mathfrak{sl}_2$-modules they are direct sum of standard representations, denote by $\mathcal{S}_1$ the category of finite-dimensional $\hat{\g}$-modules $N$ such that the action of $\alpha_{\cJ}$ induces a short grading on $N$. In \cite{KS} functors $Lie$ and $Jor$ were extended to functors between categories $\Jhalf$ and $\mathcal{S}_{\frac12}$, and $\Jone$ and $\mathcal{S}_1$. 
\begin{thm}[Proposition 2.1,\cite{KS2}]\label{Lie-Jor-one}
The functors $Lie$ and $Jor$ define an equivalence of the categories $\Jhalf$ and $\mathcal{S}_{\frac12}$. 
\end{thm}
It was proven in \cite{KS2} for superalgebras, which implies trivially the even case. For the other pair of 
categories $\Jone$ and $\mathcal{S}_1$ the relation is more intricate and can be found in Section 3 of \cite{KS}.
Since in this paper we focus on $\mathcal{S}_{\frac12}$, for the unital case we only mention 
\begin{thm} [\cite{KS}]\label{Lie-Jor-mod}
Functors $Lie: \Jone\to \mathcal{S}_1$ and $Jor: \mathcal{S}_1\to\Jone$ have the following properties
\begin{enumerate}
\item $Jor\circ Lie$ is isomorphic to the identity functor in $\Jone$.
\item If $N$ is simple non-trivial module in $\mathcal{S}_1$ then $Jor(L)$ is simple in $\Jone$.
\end{enumerate}
\end{thm}
Next we describe irreducible modules in both $\mathcal{S}_{\frac12}$ and  $\mathcal{S}_1$ for $Lie(\I)$, $\I$ a Jordan algebra
from Theorem \ref{albert-theorem}.  
\begin{enumerate}
\item[A.] $\cJ=\kk$, $Lie(\kk)=\ssl(2)$: there is only one (irreducible module) in $\mathcal{S}_{\frac12}$ -- the two-dimensional simple
  $\mathfrak{sl}(2)$-module $L$; and only one irreducible module in $\mathcal{S}_1$ -- the adjoint $\mathfrak{sl}(2)$-module $ad$.
\item[B1.] $\cJ=H_n(C_1)=Sym_n^+(\kk)$, $n\geq 3$; $Lie(Sym_n^+(\kk))=\ssp(2n)$:  there is only one module with grading of length $2$, standard $2n$-dimensional $V$; and two $\ssp(2n)$-modules with short grading, $n(2n+1)$-dimensional adjoint representation $ad=S^2 V$, and $n(2n-1)-1$-dimensional $\overline{\Lambda^2 V}=\Lambda^2 V/{\kk}$. 
\item[B2.] $\cJ=H_n(C_2)=M_n^+(\kk)$, $n\neq 3$; $Lie(M_n^+(\kk))=\ssl(2n)$: there are two modules in $\mathcal{S}_{\frac12}$ both of dimension $2n$, standard $V$ and its dual $V^*$; there are five modules
in $\mathcal{S}_1$, $4n^2-1$-dimensional adjoint module $ad$, two $n(2n+1)$-dimensional modules 
$S^2(V)$ and $S^2(V^*)$, and  two $n(2n-1)$-dimensional modules 
$\Lambda^2(V)$ and $\Lambda^2(V^*)$.
\item[B3.] $\cJ=H_n(C_3)=Symp^+_n(M_2(\kk))$, $n\geq 3$, $Lie(Symp^+_n(M_2(\kk)))=\so(4n)$: there is only one module in $\mathcal{S}_{\frac12}$ -- the standard $4n$-dimensional module $V$; and two in $\mathcal S_1$ -- the  $2n(4n-1)$-dimensional adjoint representation $ad\simeq\Lambda^2 V$, and the  $2n(4n+1)-1$-dimensional $\overline{S^2 V}=S^2 V/{\kk}$. For $n=3$, $\so(12)$ has one more module with short grading--
$32$-dimensional spinor module $\Gamma^+$.
\item[B4.] $\cJ=Alb=H_3(C_4)$, $Lie(Alb)=E_7$; there are no modules in $\mathcal{S}_{\frac12}$,
there is only one (irreducible) module in $\mathcal{S}_1$-- $133$-dimensional adjoint representation $ad$.
\item[C1.] $\cJ=\cJ(V,f)$, $\dim\cJ=2m-1$, $m\geq 2$, $Lie(\cJ(V,f))=\so(2m+1)$: there is only 
one module in $\mathcal{S}_{\frac12}$-- the $2^{m}$-dimensional spinor representation $\Gamma$; and $m$ modules in $\mathcal S_1$ -- the exterior powers of the standard module $\Lambda^r V$, $r=1,\dots,m$.
\item[C2.] $\cJ=\cJ(V,f)$, $\dim\cJ=2m-2$, $m\geq 3$, $Lie(\cJ(V,f))=\so(2m)$: 
there are two modules in $\mathcal{S}_{\frac12}$,-- the $2^{m-1}$-dimensional spinor representations 
$\Gamma^+$ and $\Gamma^-$; and $m+1$ modules in $\mathcal S_1$ -- the exterior powers of the standard module $\Lambda^r V$, $r=1,\dots,m-1$,
and $\Lambda^+$, $\Lambda^-$ (two irreducible components of the reducible $\Lambda^m V=\Lambda^+\oplus \Lambda^-$).
\end{enumerate}

Note that for a simple Lie algebra the category $\mathcal{S}_{\frac12}$ has at most two simple objects. With exception of $E_7$ any simple module  in $\mathcal S_1$ is a simple component of a tensor product of two simple modules from $\mathcal{S}_{\frac12}$.
 
\section{Category $\mathcal{S}_{\frac12}$ for Jordan algebras with radical squared zero}\label{section_representation_algebras}

Let $\cJ$ be a finite-dimensional Jordan algebra with $\Rad^2\cJ=0$, then Theorem~\ref{albert-theorem}
implies that $\cJ=\I+\R$, where $\I=\oplus_{i=1}^k \I_i$ and $\R=\oplus_{j=1}^r M_j$, and $M_j$ is an irreducible unital $\I$-module from (2) and (3) Lemma~\ref{simple_module_over_semi_simple_algebra}. 
Theorems \ref{Lie-Jor-alg}, \ref{Lie-Jor-one}  and \ref{Lie-Jor-mod} imply 
\begin{equation}\label{Lie-g}
\g=Lie(\cJ)=Lie(\I)+Lie(\R), \quad Lie(\I)=\oplus_{i=1}^k Lie(\I_i), \quad Lie(\R)=\oplus_{j=1}^r Lie(M_j),
\end{equation}
where $Lie(\R)^2=0$, $Lie(\I_j)$ is one of the Lie algebras in the A-C list of the last section and
$M_j$ is either $\I_i$-module in $\mathcal{S}_1$ from that list or tensor product of two modules in 
 $\mathcal{S}_{\frac12}$ (over different simple Lie algebras $Lie(I_i)$ and $Lie(I_J)$) from the list.
By Theorem~\ref{Lie-Jor-one} the categories $\Jhalf$ and $\mathcal{S}_{\frac12}$ are isomorphic. 
Starting from this point we forget Jordan algebras and deal with the corresponding Lie algebra modules.

Let us recall the following well-known observation usually
attributed to Gabriel, \cite{Gab2}. Let $\mathcal C$ be an abelian category
with finitely many simple modules such that every object has
finite length and every simple object has a projective cover. Let
$P$ be a projective generator of $\mathcal C$ and
$A=\operatorname{End}P$, then $\mathcal C$ is equivalent to the
category of finite-dimensional $A$-modules. If $L_1,\dots,L_r$ is
the set of all up to isomorphism simple objects in $\mathcal C$
and $P_1,\dots,P_r$ are their projective covers, then $A$ is a
pointed algebra which is usually realized as the path algebra of a
certain quiver $Q$ with relations. The vertices of $Q$ correspond
to simple (resp. projective modules) and the number of arrows from
vertex $i$ to vertex $j$ equals
$\operatorname{dim}\operatorname{Ext}^1(L_j,L_i)$ (resp.
$\operatorname{dim}\operatorname{Hom}(P_i,\operatorname{rad}P_j/\operatorname{rad}^2 P_j$).

We apply this approach to the case when $\mathcal C$ is $\cS_{\frac12}$
for a Lie algebra $\g=\g_s\ltimes\rr$ with semisimple subalgebra $\g_s$ and an abelian radical $\rr$.

\begin{lem}\label{centext} Let $\g=\g_s\ltimes\rr$, where $\rr$ is an abelian
  radical. Then $\hat\g=\g_s\oplus\rr\oplus \Lambda^2(\rr)^{\g_s}$ is a direct sum of $\g_s$-modules
  with $\mathfrak z=\Lambda^2(\rr)^{\g_s}$. The
\end{lem}
\begin{proof} Since by the Levi theorem $\hat\g=\g_s\oplus \hat\rr$,
  we have that $\mathfrak z=[\rr,\rr]$ and $\hat\rr=\rr\oplus\mathfrak z$ as a
$\g_s$-module. The statement follows.
\end{proof} 

Observe that by minimality of $(\g,\alpha)$ the radical $\rr$ does not contain trivial 
$\g_s$-submodules. Therefore $\rr$ is 
generated by $\rr_{-1}$ or by $\rr_1$
as a $\g_s$-module, where $\rr_{i}=\g_i\cap\rr$, $i=\pm 1$.

The simple objects in $\cS_{\frac12}$ are simple $\g_s$-modules with
grading of length 2. The projective cover $P(L)$ of a simple module
$L$ is a unique maximal quotient of the induced module
$$I(L)=U(\hat\g)\otimes _{U(\g_s)}L,$$
lying in $\cS_{\frac12}$. To construct this quotient consider a
$\hat\g$-submodule $N\subset I(L)$ generated by all vectors of degree
greater than $\frac12$ and less than $-\frac12$ and then set $P(L):=I(L)/N$.

It is now clear how to describe the quiver $Q$.

\begin{lem}\label{quiver} Let $L$ and $L'$ be two simple
  $\g_s$-modules.
Then $\operatorname{dim}\operatorname{Ext}^1_{\hat\g}(L,L')$ equals the
multiplicity of $L'$ in $L\otimes \rr$.
\end{lem}
\begin{proof} Consider a non-trivial extension
$$0\to L'\to M\to L\to 0.$$ Then $\rr M\subset L'$, therefore $\rr^2 M=0$ and $\mathfrak z=[\rr,\rr]$ acts trivially on $M$, thus $M$ is a $\g$-module and we have:
$$
\operatorname{Ext}^1_{\hat\g}(L,L')=\operatorname{Ext}^1_{\g}(L,L')= H^1(\g,L^*\otimes L')=
H^0(\g_s, H^1(\rr,L^*\otimes L'))=\Hom_{\g_s}(\rr,L^*\otimes L'),
$$
where the third equality follows from the spectral sequence with respect to the ideal $\rr$.
Therefore $\operatorname{Ext}_{\hat\g}^1(L,L')=\operatorname{Hom}_{\g_s}( \rr\otimes L, L')$.
\end{proof}

\begin{Cor}\label{mult1} Suppose $\rr$ is a sum of $m$ irreducible $\g_s$-modules then 
for any vertices $j,k$ of $Q$ the number of arrows from $j$ to $k$ is at most $m$.
\end{Cor}
\begin{proof} Let $L$ be a simple $\g_s$-module $L$ with grading of length 2. It follows from the list (A)--(C) that  
all weights of $L$ have multiplicities one, hence for any irreducible component $\rr'$ of $\rr $ the tensor product
  $\rr'\otimes L$ is a multiplicity free $\g_s$-module. Therefore $\dim\operatorname{Hom}_{\g_s}(L_k,L_j\otimes\rr)\leq m$.
\end{proof}

Now we consider decompositions
$$\g_s=\bigoplus_{i=1}^k \g_s^{(i)},\,\, \rr=\bigoplus_{j=1}^m\rr^{(j)},$$
where each $\g_s^{(i)}$ is a simple ideal of $\g_s$ and each $\rr^{(j)}$ is a
simple $\g_s$-module.
This decomposition completely defines the quiver $Q$. Indeed, to each 
$\g_s^{(i)}$ we associate one vertex in $Q$ if $\g_s^{(i)}=\mathfrak{sp}(2n)$,
$\mathfrak{so}(2n)$ with the first short grading or $\mathfrak{so}(2n+1)$
and two vertices if $\g_s^{(i)}=\mathfrak{sl}(2n)$ or $\mathfrak{so}(2n)$
with the second short grading. Each vertex corresponds to a simple
$\g_s^{(i)}$-module lying in $\cS_{\frac12}$. Next, for each  $\rr^{(q)}$ we define the
arrow $i\to j$ if $\operatorname{Hom}_{\g_s}(L_i, L_j\otimes \rr^{(q)})\neq 0$.
Thus, all vertices can be colored into $k$ colors and there is at most
two vertices for each color. In the similar way all arrows can be
colored into $m$ colors and as we will see below there is at most two
arrows of the same color as well.

\begin{lem}\label{wild3} Assume that for some $j_1<j_2<j_3\leq m$  
$\rr^{(j_1)}$, $\rr^{(j_2)}$ and $\rr^{(j_3)}$ are isomorphic as
  $\g_s$-modules.
Then  $\cS_{\frac12}$ is wild.
\end{lem}
\begin{proof} $Q$ contains a Kronecker quiver with the arrow of
  multiplicity $3$.
\end{proof}

We say $\rr^{(q)}$ is of type I if $\rr^{(q)}=L_1\otimes L_2$ where $L_1$ and 
$L_2$ are simple $\g_s^{(i)}$ and $\g_s^{(j)}$-modules respectively for some
$i\neq j$. In this case $L_1$ and $L_2$ are simple modules in  $\cS_{\frac12}$
and $\rr^{(q)}$ defines exactly two arrows $[L_1]\to [L_2^*]$ and
$[L_2]\to [L_1^*]$, where by $[L]$  we denote the vertex
corresponding to a simple module $[L]$.

We say $\rr^{(q)}$ is of type II if $\rr^{(q)}$ is a simple  $\g_s^{(i)}$-module
for some $i\leq k$. In this case $\rr^{(q)}$ defines either one or two arrows in
$Q$ between two vertices of the same color $i$.

It is very important that in our situation for any simple $L$ and a
component $\rr^{(p)}\subset\rr$, $\Hom_{\g_s} (S,\rr^{(p)}\otimes L)\neq 0$ for at most
one simple $S$
and similarly $\Hom_{\g_s} (L,\rr^{(p)}\otimes S')\neq 0$ for at most one
$S'$. 

Thus, our coloring of $Q$ satisfies the following conditions:
\begin{enumerate}
\item There are at most two vertices of the same color.
\item There are at most two arrows of the same color.
\item Two arrows of the same color do not have common head or tail.
\end{enumerate}

We say that $\rr^{(q)}$ is {\it singular} if
$\operatorname{dim}\rr^{(q)}_1=1$. One can see from the list (A)--(C)
if $\rr^{(q)}$ is singular of type I,
then $\rr^{(q)}\simeq E_1\otimes E_2$, where $E_1,E_2$ are the standard
two-dimensional modules over two distinct copies of $\mathfrak{sl}(2)$,
and if $\rr^{(q)}$ is singular of type II, then  $\rr^{(q)}$ is the
standard $\mathfrak{so}(n)$-module, and  $\mathfrak{so}(n)$ itself
has the second short grading.

Our next goal is to determine the relations in the quiver $Q$.
\begin{lem}\label{radical filtration} Let $L$ be a simple module in  $\cS_{\frac12}$
and $P=P(L)$ be its indecomposable projective cover. Then as a
$\g_s$-module $P$ has a direct sum decomposition
$P=L\oplus\rr L\oplus \rr^2 L\oplus \dots$ and $\rr^i(L)\simeq{\rad}^{i}P/{\rad}^{i+1}P$. 
\end{lem}
\begin{proof} Follows from the fact that ${\rad}^{i}P=(\rr+\mathfrak z)^i P$ and that
$\mathfrak z P\subset \rr^2 P$.
\end{proof}

\begin{Cor}\label{prod-two}  Let $\alpha$ and $\beta$ be two arrows in $Q$ of colors
  $p$ and $q$ respectively. Let $L$ (resp. $L'$) be the simple modules corresponding
  to the tail of $\alpha$ (resp. the head of $\beta$). 
The product $\alpha\beta\neq 0$ iff $\Hom_{\g_s}(L,\rr^{(p)}\rr^{(q)}L')\neq 0$. 
\end{Cor}

\begin{lem}\label{whenzero} Let $\alpha$ and $\beta$ be two arrows in $Q$ of colors
  $p$ and $q$ respectively.
If $\alpha\beta\neq 0$, then $\rr^{(p)}\simeq(\rr^{(q)})^*$ as $\g_s$-modules.
Moreover, in the case $p=q$ the invariant pairing 
$\rr^{(p)}\times\rr^{(p)}\to \kk$ is skew-symmetric.
\end{lem}
\begin{proof} Suppose that $\rr^{(p)}$ and $(\rr^{(q)})^*$ are not isomorphic.
In view of Corollary \ref{prod-two} it suffices to show that for any $M\in\cS_{\frac12}$ we have
$\rr^{(p)}\rr^{(q)} M=0$. By Lemma \ref{centext} $xy=yx$ for any $x\in\rr^{(p)}$
and $y\in \rr^{(q)}$.

Let $N$ be the subset of vectors in $M$ annihilated by $\rr^{(p)}$. Then
$N$ is a $\g_s$-submodule and $N=N_{1/2}\oplus N_{-1/2}$.
We claim that 
$$N_{1/2}=\{m\in M_{1/2} | \rr^{(p)}_{-1}(m)=0\}.$$
Indeed, 
$$\rr^{(p)}_{0}m=[(\g_s)_1,\rr^{(p)}_{-1}]m=(\g_s)_1\rr^{(p)}_{-1}m+\rr^{(p)}_{-1}(\g_s)_1m=0$$
and $\rr^{(p)}_{1}m=0$ because of the grading.
In the same way,
$$N_{-1/2}=\{m\in M_{-1/2} | \rr^{(p)}_{1}(m)=0\}.$$

Let us show that $\rr^{(q)}(M)\subset N$.
Indeed,
$$\rr^{(p)}_{-1}\rr^{(q)}_{1}(M_{-1/2})=\rr^{(q)}_{1}\rr^{(p)}_{-1}(M_{-1/2})=0.$$
Therefore $\rr^{(q)}_{1}(M_{-1/2})\subset N$. But
$\rr^{(q)}_{1}(M_{1/2})=0$. Thus,  $\rr^{(q)}_{1}(M)\subset N$. Since $N$ is $\g_s$-invariant
and $\rr^{(q)}=[\g_s,\rr^{(q)}_1]+[\g_s,[\g_s,\rr^{(q)}_1]]$ we have
$$[\g_s,\rr^{(q)}_{1}](M)=\g_s\rr^{(q)}_{1}(M)+\rr^{(q)}_{1}\g_s(M)\subset N,$$
$$[\g_s,[\g_s,\rr^{(q)}_{1}]](M)=\g_s[\g_s,\rr^{(q)}_{1}](M)+[\g_s,\rr_1^{(q)}] \g_s(M)\subset N,$$
and hence $\rr^{(q)}(M)\subset N$.

For the last assertion just observe that in the case $p=q$ 
by Lemma \ref{centext} we have $xy=yx$ for any $x,y\in\rr^p$ if $\Lambda^2(\rr^{(p)})^{\g_s}=0$.
\end{proof}

\begin{Cor}\label{squarezero} Assume that $\hat\g=\g$. If $M$ is a
  $\g$-module in the category  $\cS_{\frac12}$ then $\rr^2 M=0$.
\end{Cor}

\begin{Cor} If $\hat\g=\g$, then the product of two arrows $i\to j$
  and $j\to k$ in the  quiver $Q$ is zero.
\end{Cor}

\begin{lem}\label{square-copies-of-module} If $P$ is the indecomposable projective cover of a simple
  module $L$  in $\cS_{\frac12}$, then $\operatorname{rad}^2P/\operatorname{rad}^3P\simeq\rr^2 L$
is isomorphic to a sum of several copies of $L$.
\end{lem}
\begin{proof} Consider $\bar{P}=P/\mathfrak z P$. Then by Lemma~\ref{whenzero}
$\rr^2\bar{P}=0$,  equivalently  $\rr^2 P\subset \mathfrak zP$.
Thus the statement follows from Lemma \ref{radical filtration}.
\end{proof}

\begin{Cor}\label{prod-two-more}  Let $\alpha$ and $\beta$ be two arrows in $Q$. If
$\alpha\beta\neq 0$, then the head of $\beta$ coincides with the tail
of $\alpha$. In other words the corresponding path is a cycle.  
\end{Cor}

\begin{lem}\label{singularity} Suppose that $\alpha$ and $\beta$ are two arrows in $Q$ of
  colors $p$ and $q$ respectively
  such that $\alpha\beta\neq 0$.
\begin{enumerate}
\item[(a)]\label{one_of_them_sl_2} If $\rr^{(q)}$ and therefore $\rr^{(p)}$ is of type I then
  $\rr^{(q)}=L\otimes S$, where $L$ is the standard module over a
  simple ideal isomorphic to $\mathfrak{sl}(2)$ and $S$ is some
    simple module lying in $\cS_{\frac12}$. Moreover,
$L$ is isomorphic to the simple module standing at the head of $\beta$.

\item[(b)]\label{radical_1_of_dim_1} If $\rr^{(q)}$ and therefore
  $\rr^{(p)}$ is of type II,
  then they are singular.
\end{enumerate}
\end{lem}
\begin{proof} By Corollary \ref{prod-two-more} the tail of $\alpha$
coincides with the head of $\beta$, thus by Lemma~\ref{whenzero} $(\rr^{(q)})^*\simeq \rr^{(p)}$ and there exists non-degenerate $\g_s$-invariant pairing $\lambda:\, \rr^{(p)}\times \rr^{(q)}\to \kk $. 
Let $\beta:[S]\to [L]$ and $\alpha:[L]\to[S]$ for some simple $S$ and $L$ and 
let $\varphi:\rr^{(p)}\otimes S\to L$ and
$\psi:\rr^{(q)}\otimes L\to S$ be the corresponding nonzero morphisms of $\g_s$-modules as in Lemma~\ref{quiver}.
The condition  $\alpha\beta\neq 0$ implies 
$$\varphi(y\otimes\psi(x\otimes l))\neq 0$$
for some $x\in\rr^{(q)},y\in\rr^{(p)},l\in L$.
Suppose that $l\in L_{1/2}$, $y\in\rr^{(p)}_1$ and $x\in\rr^{(q)}_{-1}$,
then $xyl=0$ and $yxl=[y,x]l=\lambda(y,x)l$, 
hence we have
\begin{equation}\label{v1}
\varphi(y\otimes\psi(x\otimes l))=\lambda(y,x)l.
\end{equation}
We claim that the
restriction $$\psi:\rr^{(q)}_{-1}\otimes L_{1/2}\to S_{-1/2}$$
is injective. Indeed, assume this is not true. Then there exist
linearly independent $x_1,\dots,x_n\in \rr^{(q)}_{-1}$ and linearly
independent $l_1,\dots,l_n\in L_{1/2}$ such that
$$\psi(\sum_{i=1}^n x_i\otimes l_i)=0.$$
Let us choose $y\in \rr^{(p)}_1$ such that $\lambda(y,x_1)=1$, 
$\lambda(y,x_i)=0$ for $i>1$ (such $y$ exists since the restriction of $\lambda$ onto $\rr^{(p)}_{-1}\times \rr^{(q)}_{1}$ is non-degenerate). Then (\ref{v1}) implies
$$0=\varphi(y\otimes\psi(\sum_{i=1}^n x_i\otimes l_i))=\sum_{i=1}^n \varphi(y\otimes\psi(x_i\otimes l_i))=\lambda(x_1,y)l_1=l_1.$$
Contradiction.
Next we consider two cases. 
\begin{enumerate}
\item[(a)] If $\rr^{(q)}$ is of type I it isomorphic to $L^*\otimes S$
  and the map $L^*_{-1/2}\otimes S_{-1/2}\otimes L_{1/2}\to S_{-1/2}$
can be injective only in case when $\dim L_{1/2}=1$. That implies the
first assertion.
\item[(b)]  If $\rr^{(q)}$ is of type II then it is a module of length $3$ over 
some simple ideal of $\g_s$ and
both $S$ and $L$ are simple modules over the same ideal. Therefore
$S=L$ or $S=L^*$. We have $\dim L_{1/2}=\dim S_{-1/2}$, hence 
injectivity of  $$\psi:\rr^{(q)}_{-1}\otimes L_{1/2}\to S_{-1/2}$$
implies $\dim\rr^{(q)}_{-1}=\dim\rr^{(q)}_{1}=1$.
\end{enumerate}\end{proof}

\begin{lem}\label{prod-two-samecolor}  Let $\alpha$ and $\beta$ be two
  arrows in $Q$ of the same color $p$. If
$\alpha\beta\neq 0$, then $\rr^{(p)}$ is of type I and is isomorphic to
$L\otimes S$ where $L$ is the standard
  $\mathfrak{sl}(2)$-module
and $S$ is a simple module that has an invariant symmetric form.
\end{lem}
\begin{proof} Follows from  the observation that by Lemma \ref{whenzero}
$\rr^{(p)}$ must have an invariant skew-symmetric form. By direct
  inspection of the list that implies  $\rr^{(p)}$ is of type I. Now
  the statement follows from the previous lemma.
\end{proof}
\begin{lem}\label{non-zero case}  Let $\g_s=\mathfrak{sl}(2)\oplus\g'$, where $\g'$ is some simple ideal,  
and let $\rr= L\otimes S$, where  $L$ is the simple standard $\mathfrak{sl}(2)$-module and $S$ is a simple 
$\g'$-module with very short grading. Assume that $S$ admits $\g'$-invariant symmetric form $B$. 
Then $Q$ has the connected component
 \begin{equation}\label{algebraA1} 
 \xymatrix{\bullet\ar@/^0.4pc/[rr]^{\alpha} & &
\bullet\ar@/^0.4pc/[ll]^{\beta}} \qquad {\text with\ relation} \qquad \alpha\beta=0.
\end{equation}
The left vertex of $Q$ corresponds to $L$ and the right to $S$. 
\end{lem}
\begin{proof}  Since $\rr=L\otimes S$, $L^*=L$ and $S^*=S$, the quiver of $\g$ contains two arrows $[L]\to [S]$ and $[S]\to [L]$. Moreover 
since $S$  admits $\g'$-invariant symmetric form, $\g'$ can not be $\mathfrak{sl}_2$ and therefore $\alpha\beta=0$. Let
$M:=L\oplus S\oplus L$, and define the action of $\rr$ on $M$
  $$(a\otimes b)(x,y,w):=(0,\omega(a,x)b, B(b,y)a)$$
  where $a,x,w\in L, \,b,y\in S$ and $\omega$ is an invariant symplectic form on $L$.
  Then for $a'\in L$, $b'\in S$ we have
  $$[a\otimes b,a'\otimes b'](x,y,w)=(0,0,B(b,b')(\omega(a,x)a'-\omega(a',x)a)).$$
  Choose the central element $z\in\mathfrak{z}$ so that
  $$[a\otimes b,a'\otimes b']=\omega(a,a')B(b,b')z,$$
  and set $z(x,y,w):=x$. Then indeed $M$ is a module over $\hat \g$ since
  \begin{equation}\label{standard-relation}
  -\omega(a,x)a'+\omega (a',x)a+\omega(a,a')x=0.
  \end{equation}
  The relation~\eqref{standard-relation} follows from the fact that it is skew-symmetric in $a$, $x$ and $a'$ and $\Lambda^3 L=0$.
  From the other side, since $\alpha\beta=0$ the indecomposable projective cover $P(L)$ has length three and therefore $P(L)\simeq M$.
  
  \end{proof}
\noindent Recall that $\mathcal A$ denotes the pointed algebra of the projective generator of $\cS_{\frac12}$. We denote the pointed algebra corresponding to the quiver with relations \eqref{algebraA1} by $\mathcal A_1$.

 \begin{lem}\label{non-zero case2} Let $\g_s=\mathfrak{sl}(2)\oplus\g'$ where $\g'$ is some simple ideal and $\rr=L\otimes S \oplus L\otimes S^*$ where $L$ is the standard $\mathfrak{sl}(2)$-module
  and $S$ is a simple $\g'$-module with very short grading and $S$ is not isomorphic to $S^*$.
   Then $Q$ is
 \begin{equation}\label{algebraA2}
 \xymatrix{\bullet\ar@/^0.4pc/[rr]^{\alpha} & &
\bullet\ar@{-->}@/^0.4pc/[ll]^{\beta}\ar@/^0.4pc/[rr]^{\delta} & & \bullet\ar@{-->}@/^0.4pc/[ll]^{\gamma}} \quad 
\text {with\ relations} \quad \beta\alpha=\delta\gamma=\beta\gamma=\delta\alpha=0,\ \alpha\beta=\gamma\delta.\end{equation}
The middle vertex corresponds to $L$
\end{lem}
\begin{proof} The proof is very similar to one of Lemma \ref{non-zero case}. The quiver of $\g$ has three vertices $[L]$, $[S]$ and $[S^*]$ 
and two pairs of arrows $[L]\to[S^*]$, $[S]\to [L]$ corresponding to $L\otimes S$ and $[L]\to[S]$, $[S^*]\to [L]$ corresponding to $L\otimes S^*$.
All zero relations in \eqref{algebraA2} follow from the fact that $\g'$ is not $\mathfrak{sl}_2$. To show that there is non-zero relation in vertex $[L]$,
define $M:=L\oplus (S\oplus S^*)\oplus L$ by the same formulas as in the proof of previous lemma using the symmetric form $B$ on 
$S\oplus S^*$ defined by
  $$B(x\oplus y,x'\oplus y')=y'(x)+y(x').$$
Then $M$ is a projective cover of $L$.
\end{proof}
\noindent We denote the pointed algebra corresponding to the quiver with relations \eqref{algebraA2} by $\mathcal A_2$.

Recall that the {\it Segre product} of two graded algebras $A$ and $B$ is the graded algebra $A\circ B=\oplus_{n=0}^{\infty} A_n\otimes B_n$.
The grading on the pointed algebras  $\mathcal A_1$ and $\mathcal A_2$ is given by length of the path.

\begin{prop}\label{mutiplicity} Let $\g_s=\mathfrak{sl}(2)\oplus\g'$ where $\g'$ is some simple ideal not isomorphic to $\mathfrak{sl}(2)$. Let $S$ be  a simple $\g'$-module with very short grading and $L$ be
  the simple $2$-dimensional $\mathfrak{sl}(2)$-module. 

(a) Assume that $\rr=W\otimes L\otimes S$ and $S$ admits a $\g'$-invariant symmetric form.  Then $\mathcal A$ of $\mathcal{S}_{\frac12}$ 
is $\mathcal{A}_1\circ S(W)$.

(b)  Assume that $\rr=W\otimes L\otimes S$ and $S$ admits a $\g'$-invariant skew-symmetric form.  Then $\mathcal A$ of $\mathcal{S}_{\frac12}$ 
is $\mathcal{A}_1\circ \Lambda(W)$.

(c) Assume that $\rr=W\otimes L\otimes S\oplus W'\otimes L\otimes S^*$ with $S$ not isomorphic to $S^*$, the action of $\g_s$ on $W$ and $W'$ is trivial. Then the pointed algebra $\mathcal A$ of $\mathcal{S}_{\frac12}$ is isomorphic to the Segre product $\mathcal{A}_2\circ S(W\oplus W')$ quotient by the ideal $(\alpha\otimes W',\delta\otimes W',\beta\otimes W,\gamma\otimes W)$.

\end{prop}
\begin{proof} (a) It is clear that $\mathcal A$ is generated by $\alpha\otimes W$ and $\beta\otimes W$. Furthermore,
  $$\hat\g=\g_s\oplus W\otimes L\otimes S\oplus S^2(W),$$
  with the commutator $$[x\otimes y\otimes w,x'\otimes y'\otimes w']:=\omega(x,x')B (y,y') w\cdot w',$$
  where $x,x'\in L$, $y,y'\in S$, $w,w'\in W$, $\omega(\cdot,\cdot)$ denotes the skew-symmetric form on $L$, $B(\cdot,\cdot)$ the symmetric form on $S$ and $\cdot$ the product in $S(W)$.

  Let us describe $P(S)$ and $P(L)$. Set
  $$P(S)=S\oplus W\otimes L,\ P(L)=L\oplus W\otimes S\oplus S^2(W)\otimes L.$$
  The action of $\hat \g$ given by
  $$w\otimes x\otimes y (s):=B(y,s) w\otimes x, \quad w\otimes x\otimes y (l):=\omega(x,l) w\otimes y,$$
  $$w\otimes x\otimes y (u\otimes s):=B(y,s) w\cdot u\otimes x,$$
  for $x,l\in L$, $y,s\in S$ and $w,u\in W$
  for any $x,x'\in L$, $y,y'\in S$, $w\in W$ and $u\in S^i(W)$. Finally, $S^2(W):L\to S^2W\otimes L$ defines the action of the center $S^2W$.
 
  It follows from Lemma \ref{non-zero case} that $P(S)$ has Loewy length $2$ and $P(L)$ Loewy length 3. Projectivity of the above constructed modules
  can be easily proven from the dimension calculations. One can easily check that $\mathcal A$ is generated by $\alpha\otimes W$
  and $\beta\otimes W$ with relations $$(\alpha\otimes w)(\beta\otimes w')=0,\ (\beta\otimes w)(\alpha\otimes w')=(\beta\otimes w')(\alpha\otimes w).$$
  That proves (a).

  The proof of (b) is very similar to  above proof of (a). Let $\eta$ denote the skew symmetric form on $W$. Now 
 $$\hat\g=\g_s\oplus W\otimes L\otimes S\oplus \Lambda^2(W),$$
 with the commutator $$[x\otimes y\otimes w,x'\otimes y'\otimes w']:=\omega(x,x')\eta(y,y') w\wedge w'.$$
 We have 
$$P(S)=S\oplus W\otimes L,\ P(L)=L\oplus W\otimes S\oplus \Lambda^2(W)\otimes L,$$
 $$w\otimes x\otimes y(s):=\eta(y,s) w\otimes x, \quad w\otimes x\otimes y(l):=\omega(x,l) w\otimes y,$$
 $$w\otimes x\otimes y (u\otimes s):=\eta(y,s) w\wedge u\otimes x.$$
 One concludes that $\mathcal A$ is generated by $\alpha\otimes W$
  and $\beta\otimes W$ with relations $$(\alpha\otimes w)(\beta\otimes w')=0,\ (\beta\otimes w)(\alpha\otimes w')=-(\beta\otimes w')(\alpha\otimes w).$$ 

  Now let us prove (c). We have $$\hat\g=\g_s\oplus (W\otimes L\otimes S\oplus W'\otimes L\otimes S^*)\oplus (W\otimes W')$$ with central extension
  $$[x\otimes y\otimes w,x'\otimes y'\otimes w']=\omega(x,x') y'(y) w\otimes w',$$
  where $x,\, x'\in L$, $y\in S$, $y'\in S^*$, $w\in W$, $w'\in W'$.
  As in the previous cases we define the projective $P(S)$, $P(L)$ and $P(S^*)$:
  $$P(S)=S\oplus W'\otimes L,\quad P(S^*)= S^*\oplus W\otimes L,$$
  $$P(L)=L\oplus (W\otimes S\oplus W'\otimes S^*)\oplus W\otimes W'\otimes L.$$
  Define the action of $\rr$  by
  $$(w'\otimes x\otimes y')(s):=y'(s)(w'\otimes x),\quad (w\otimes x\otimes y)(s'):=s'(y)(w\otimes x),$$
  $$(w'\otimes x_1\otimes y'+w\otimes x_2\otimes y)(l):=\omega(x_1,l)w'\otimes y'+\omega(x_2,l) w\otimes y,  $$
  $$(w'\otimes x_1\otimes y'+w\otimes x_2\otimes y)(u\otimes s+u'\otimes s'):=y'(s)(w'\otimes u\otimes x_1)-s'(y)(u'\otimes w\otimes x_2).$$
  Here $w,u\in W$, $w',u'\in W'$, $x,x_1,x_2,l\in L$, $y,s\in S$, $y',s'\in S^*$.
  Checking that the above modules are indeed projective can be done using Lemma \ref{non-zero case2}. Then as in the previous cases
  $\mathcal A$ is generated by $$\alpha\otimes W,\delta\otimes W,\beta\otimes W',\gamma\otimes W'$$ in the Segre product
  $\mathcal{A}_2\circ S(W\oplus W')$.
  \end{proof}

  \begin{lem}\label{Clifford example} Let $\g_s=\mathfrak{so}(m)$  with the second short grading and $\rr=W\otimes V$, where $V$ is the standard $\g_s$-module and 
      the action of $\g_s$ on $W$ is trivial.

      (a) If $m=2n$, then the category $\cS_{\frac12}$ is equivalent to the category of $\mathbb Z_2$-graded $\Lambda(W)$-modules.
      
      (b) If $m=2n+1$, then the category $\cS_{\frac12}$ is equivalent to the category of $\Lambda(W)$-modules.
    \end{lem}
    \begin{proof} Recall that       
$\hat{\g}=\g_s\oplus \rr\oplus \Lambda^2(W)$
and the commutator on $\rr$ is defined by 
$$[w\otimes v, w'\otimes v']=(v,v')w\wedge w'.$$

      There exists a homomorphism $\rho: U(\g_s)\to C(V)$ where $C(V):=T(V)/(\{v,v'\}-(v,v'))_{v,v'\in V}$ is the Clifford algebra on the space $V$.
      If $m=2n$ then $C(V)$ has exactly one up to isomorphism irreducible representation $\Gamma$ which after restriction to $\g_s$ is isomorphic to the direct sum $\Gamma^+\oplus\Gamma^-$.
      If $m=2n+1$ then $C(V)$ has two irreducible representations $\Gamma$ and $\Gamma'$ and either of them is isomorphic to the spinor representations when restricted to $\g_s$.
      We extend now $\rho$ to a homomorphism $\tilde\rho:U(\hat\g)\to \Lambda(W)\otimes C(V)$ by setting
      $$\tilde\rho(X)=1\otimes\rho(X),\,\text{for}\, X\in\g_s,\quad \tilde\rho(w\otimes v)=w\otimes v,\,\text{for}\, w\otimes v\in\rr.$$
      Then clearly, $P:=\Lambda(W)\otimes\Gamma$ is a  $\Lambda(W)\otimes C(V)$-module and the pull back of it gives a $\hat\g$-module. Furthermore, it follows from construction that this module is indecomposable with cosocle isomorphic to $\Gamma$.

      Next, we will prove that $P$ is projective. Note that $P=\tilde\rho(U(\hat\rr))\Gamma$ and the projective generator of $\cS_{\frac12}$ is a quotient of $U(\hat\rr)\otimes\Gamma$.
      Let $I=\ker\tilde\rho\cap U(\hat\rr)$.
      Therefore it suffices to check that $I$ acts trivially on any $M\in \cS_{\frac12}$. First we claim that we have the following relations on any $M\in \cS_{\frac12}$.
      \begin{enumerate}
      \item $(w\otimes v)(w\otimes v')=0$ for any $w\in W$ and $v,v'\in V$;
      \item  $(w\otimes v)(w'\otimes v)=0$ for any $w,w'\in W$ and $v\in V$ such that $(v,v)=0$;
        \item $(w\otimes v)(w'\otimes v')=-(w'\otimes v)(w\otimes v')$ for any $w,w'\in W$ and $v,v'\in V$;
        \item $(w\wedge w')(w\otimes v)=0$ for any $w,w'\in W$ and $v\in V$;
          \item $(w\wedge w')(w''\otimes v)=-(w''\wedge w')(w\otimes v)$ for any $w,w',w''\in W$ and $v\in V$;
        \item $(w\wedge w')(w\wedge w'')=0$ for any $w,w',w''\in W$.
        \end{enumerate}
        Indeed, the first relation is a consequence of Lemma \ref{prod-two-samecolor}. The second follows from the grading argument for $v_1\in V_1$ and from the fact that any isotropic $v$ lies on the orbit $SO(V)v_1$.
        From (1) we get (3) using
        $$((w+w')\otimes v)((w+w')\otimes v')=0.$$
        To obtain (4) we use
        $$(w\wedge w')(w\otimes v)=[w\otimes v',w'\otimes v](w\otimes v)=0$$
        for $(v,v)=0$, $(v,v')=1$. This implies (4) for isotropic $v\in V$ and hence for all $v\in V$ by linearity. To prove (5)
        $$0=((w+w'')\wedge w')((w+w'')\otimes v)=(w\wedge w')(w''\otimes v)+(w''\wedge w')(w\otimes v).$$
        Finally for (6) take $v\in V$ such that $(v,v)=1$. Then using (4) and (3) we obtain
        $$(w\wedge w')(w\wedge w'')=(w\wedge w')[w\otimes v,w''\otimes v]=(w\wedge w')(w\otimes v)(w''\otimes v)-(w\wedge w')(w''\otimes v)(w\otimes v)=$$
        $$(w\wedge w')(w\otimes v)(w''\otimes v)+(w\wedge w')(w\otimes v)(w''\otimes v)=0.$$

        Now we claim that the ideal $J$ generated by relations (1)-(6) equals $I$. Note that $J\subset I$.
        We choose a basis $\{v_1,\dots,v_m\}$ of $V$ consisting of isotropic vectors and a basis $\{w_1,\dots,w_p\}$ of $W$. By the PBW theorem
        the isomorphism $U(\hat\rr)\cong S(W\otimes V)\otimes S(\Lambda^2(W))$ (of vector spaces). First, (6) implies $S(\Lambda^2(W))/(J\cap S(\Lambda^2(W))=\Lambda^{even}(W)$. The relations (1)-(3) imply that
        $S(W\otimes V)$ is generated by monomial $(w_{i_1}\otimes v_{j_1})\dots (w_{i_k}\otimes v_{j_k})$ for some $i_1<\dots<i_k$ and $j_1<\dots<j_k$. Using relation (5) we get $U(\hat\rr)/J$ is spanned by
        $(w_{i_1}\otimes v_{j_1})\dots (w_{i_k}\otimes v_{j_k})\otimes (w_{i_{k+1}}\wedge\dots\wedge w_{i_s})$ for even $s-k\geq 0$, $i_1<\dots<i_s$ and $j_1<\dots<j_k$. This implies $\dim U(\hat\rr)/J\leq\dim \tilde\rho(U(\hat\rr))$ and
        therefore $I=J$.

        Since $P$ is projective, the condition on the cosocle of $P$ implies that it is a projective generator of $\cS_{\frac12}$. If $m$ is odd then $P$ is an indecomposable projective and
        $\End(P)=\Lambda (W)$ as easily follows from the fact that $\Lambda(W)$ is the center of  $\Lambda(W)\otimes C(V)$.
        If $m$ is even  then $P=P^+\oplus P^-$,  $\End(P^\pm)=\Lambda^{even}(W)$, $\Hom(P^+,P^-)=\Hom(P^-,P^+)=\Lambda^{odd}(W)$.
        Thus, if $\pi^{\pm}$ are the idempotents corresponding to $P^{\pm}$, then
        $$\End(P)=\pi^+\Lambda^{even}\pi^+\oplus\pi^-\Lambda^{even}\pi^-\oplus \pi^-\Lambda^{odd}\pi^+\oplus\pi^+\Lambda^{odd}\pi^-.$$ The category of representations of this algebra is equivalent to the category
        of $\mathbb Z_2$-graded $\Lambda(W)$-modules.
\end{proof}
\begin{Rem} Note that Lemma \ref{Clifford example} covers also the case $\g_s=\mathfrak{sl}(2)\oplus \mathfrak{sl}(2)$ and $\rr=W\otimes L\otimes L'$, for the standard $\mathfrak{sl}(2)$-modules $L$ and $L'$,
  since $\mathfrak{so}(4)\simeq\mathfrak{sl}(2)\oplus \mathfrak{sl}(2)$ and $L\otimes L'$ is the standard $\mathfrak{so}(4)$-module.
  \end{Rem}
 \begin{Cor}\label{non-zero case3} Let $\g_s=\mathfrak{so}(n)$ with the second grading and $\rr=V\oplus V$ where $V$ is the standard $\g$-module.
  
   (a) If $n$ is odd then the category $\cS_{\frac12}$ is equivalent to the category of finite-dimensional representations of $\Lambda(\xi_1,\xi_2)$
$$
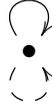
\xymatrix{\bullet\ar@{->}@(ul,ur)\ar@{-->}@(dl,dr)}
$$
\medskip

   (b) If $n$ is even then $Q$ is the quiver
  $$
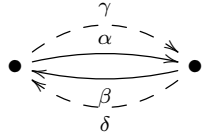
\xymatrix{\bullet\ar@/^0.4pc/[rr]^{\alpha} \ar@{-->}@/^1.3pc/[rr]^{\gamma}& &
\bullet\ar@/^0.4pc/[ll]^{\beta}\ar@{-->}@/^1.3pc/[ll]^{\delta}}$$
with relations $$\beta\alpha=\delta\gamma,\ \alpha\beta=\gamma\delta,\ \beta\gamma=\gamma\beta=\delta\alpha=\gamma\delta=0$$ 
\end{Cor}
\begin{proof} Follows from Lemma \ref{Clifford example} for the case $\dim W=2$.
\end{proof}
Let $A=\bigoplus_{i\geq 0} A_i$ and $B=\bigoplus_{i\geq 0} B_i$ are two graded algebras such that $A_0=B_0$. Consider the algebra
 $A\Pi B$ defined by
 $$(A\Pi B)_0=A_0=B_0,\ (A\Pi B)_i=A_i\oplus B_i, \ \forall \ i>0,  A_iB_j=B_jA_i=0\ \forall\ i,j>0.$$
 If $A$ and $B$ are quadratic then that $A\Pi B$ is quadratic and if  $A$ and $B$ are Koszul then $A\Pi B$ is Koszul, see \cite{PP} chapter 3.
 \begin{prop}\label{product} Let $\mathcal A$ be the pointed algebra Morita equivalent to the category $\mathcal \cS_{\frac12}$. Then $\mathcal A$ is 
   isomorphic to the product  $\prod_{i=1}^q \mathcal A^i$ where each $\mathcal A^i$ is isomorphic to $\kk^{l_i}\times \mathcal B^i$
   where $\mathcal B^i$ is either the algebra Morita equivalent to $\g_s\oplus \rr$ with irreducible $\g_s$-module $\rr$ or $\mathcal B^i$ is isomorphic to one of the algebras described in Proposition \ref{mutiplicity} or Lemma \ref{Clifford example}. In both cases $l_i$ is the number of simple objects in
   $\mathcal \cS_{\frac12}$ minus the number of simple $\mathcal B^i$-modules. 
 \end{prop}
 \begin{proof}  Let us decompose $\rr=\bigoplus^q_{i=1}\mathfrak{u}^{(i)}$ where each $\mathfrak{u}^{(i)}$ is a simple $\g_s$-module which does not satisfy Lemma \ref{singularity}, or $\mathfrak{u}^{(i)}$ is as in Proposition \ref{mutiplicity} or Lemma \ref{Clifford example} . The set of edges of the quiver of $\mathcal \cS_{\frac12}$ is the union of edges of
    $\mathcal B^i$. It follows from Lemma \ref{singularity} that any path containing edges from $\mathcal B^i$ and $\mathcal B^j$ for $i\neq j$ equals zero. The statement follows
   \end{proof}
  \begin{thm} The category $\mathcal \cS_{\frac12}$ is equivalent to the category of representations of finite-dimensional graded Koszul algebra
    $\mathcal A$.
Moreover, $\mathcal A$ is Koszul.
\end{thm}
\begin{proof} By all above computations the relations of the quiver $Q$ are quadratic and moreover all non-trivial relations are for loops or cycles of length $2$. That allows to define the grading on $\mathcal A$ with degree
  given by the
  length of the path.

  To prove that $\mathcal A$ is Koszul it suffices to check that every $\mathcal B^i$ is Koszul as follows from Proposition \ref{product}.  
  In the cases when $\mathcal B^i$ is not as in Proposition \ref{mutiplicity} or Lemma \ref{Clifford example} the radical of $\mathcal B^i$ is zero and the claim is trivial. 

  In the case of Lemma \ref{Clifford example}(b) the algebra $\mathcal B^i$ is isomorphic to the exterior algebra $\Lambda(W)$ and therefore it is Koszul.

  In the case of Lemma \ref{Clifford example}(a) $\mathcal B^i$ is isomorphic to $\Lambda(W)\circ \kk[Q]$ where  $\circ$ stands for Segre product and $\kk[Q]$ is the path algebra of
   $$
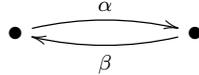
\xymatrix{\bullet\ar@/^0.4pc/[rr]^{\alpha} & &
\bullet\ar@/^0.4pc/[ll]^{\beta}}$$  By Proposition 2.1 in chapter 3 \cite{PP} $\mathcal B^i$ is a Koszul. 

For the cases in Proposition\ref{mutiplicity}(a)(respectively,(b)) $\mathcal B^i$ is isomorphic to $\Lambda(W)\circ\kk[Q']$ (respectively, $S(W)\circ\kk[Q']$), where  $\kk[Q']=\kk[Q]/(\alpha\beta)$. By the same reason as in the
previous case $\mathcal B$ is a Koszul.

Finally we will prove that in  the case of Proposition \ref{mutiplicity}(c) $\mathcal B^i$ is Koszul by a direct computation. Enumerate the vertices of the
quiver from left to right. Denote by $e_1,e_2,e_3$ the corresponding idempotents. Let $k=\dim W$ and $l=\dim W'$. One has the following exact sequences (compatible with grading): 
$$0\to(\kk e_2)^l \to \mathcal Be_1\to \kk e_1\to 0,\quad 0\to (\kk e_2)^k\to\mathcal Be_3\to \kk e_3\to 0,$$
$$0\to (\kk e_2)^{kl}\to (\mathcal Be_1)^k\oplus (\mathcal Be_3)^l\to \mathcal Be_2\to \kk e_2\to 0.$$
One can construct graded resolutions of simple modules $\kk e_i$ for $i=1,2,3$ by gluing these exact sequences.
\end{proof}
\begin{prop} Let $M$ be a $\hat\g$-module in the category $\mathcal{S}_{\frac{1}{2}}$.
  Suppose that the
multiplicity of any singular $\g_s$-component in $\rr$ is less than
$3$. Then $\hat\rr^3 M=0$.
\end{prop}
\begin{proof} Let $\mathfrak z$ denote the center of $\hat\g$. We will
show that $\hat{\rr}\mathfrak z M=0$. Let
$z=[x,y]\in\mathfrak z$ for
some
$x\in \rr^{(p)}_1,y\in\rr^{(q)}_{-1}$. Under our assumption for any $s\neq p,q$ there exists
$w\in\rr^{(s)}_1$ such that $[w,y]=0$ or $u\in\rr^{(s)}_{-1}$ such that $[u,x]=0$.
In the first case we have
$$zwM=(xyw-yxw)M=(xwy-yxw)M=0,$$
as $xwM=0$ by the grading argument.
Thus, we have $z\rr^{(s)} M=0$ by $\g_s$-invariance.
The second case can be treated similarly.
Thus, we have $z\rr^{(s)} M$ for any $s\neq p,q$.

Suppose $s=p$. Then 
using
$x^2=0$ we obtain for any $m\in M$
$$xzm=x(xy-yx)m=-xyxm=-(yx+z)xm=-zxm=zxm=0.$$
Then again by $\g_s$ invariance we have $\rr^{(p)} z M=0$.
Similarly one proves that  $\rr^{(q)} z M=0$.
Thus, we obtain $\rr z M=0$.

Any $z'\in\mathfrak z$ is a linear combination of elements satisfying
our initial assumption. Hence we have $\rr \mathfrak z M=0$.
\end{proof}

\begin{Cor} For a tame category  $\cS_{\frac12}$ a path in $Q$ of length
greater than $2$ is always zero.
\end{Cor}

\section{Building blocks}\label{section_building_blocks} 
In this section we will give a description of quiver with relations for the category $\cS_{\frac12}$ for Lie algebra
\begin{equation}\label{decomposition}
\g=\g_s\ltimes\rr, \quad \g_s=\bigoplus_{i=1}^k \g_s^{(i)},\,\, \rr=\bigoplus_{j=1}^m\rr^{(j)}
\end{equation}
Recall that by Proposition~\ref{product} any pointed algebra $\mathcal A\,$  Morita equivalent to the category $\cS_{\frac12}$ is 
   $\prod_{i=1}^q \mathcal A^i$ where each $\mathcal A^i$ is isomorphic to $\kk^{l_i}\times \mathcal B^i$. Algebra $\mathcal B^i$
   which we will call {\it building block} is either $\g_s\oplus \rr$ with irreducible $\g_s$-module $\rr$ or one of the algebras described in 
   Proposition \ref{mutiplicity} or Lemma \ref{Clifford example}. In this section we will describe $\mathcal B^i$.

To differentiate different short gradings on $\so(m)$, denote by $\so^1(4n)$ the first short grading corresponding to Jordan Hermitian algebra and 
by $\so^2(m)$ the second short grading corresponding to Jordan algebra of bilinear form.
\subsection{Building blocks with simple radical}
\begin{ex}
Let $\g=\so(12)\ltimes\Gamma^+$. Then Lemma~\ref{lemma-tensor} (3) implies that $\cS_{\frac12}$ is simple and coincide with 
$\cS_{\frac12}$ for $\so(12)$. 
\end{ex}
\noindent This example implies that in decomposition \eqref{decomposition} for $\rr$ we may drop any component $\rr^{(j)}$ isomorphic to
$\Gamma^+$. 

In the following two tables we describe $Q$ for $\g=\g_s\ltimes\rr$, $\g_s$ simple algebra or $\g_s=\g_1\oplus\g_2$ sum of two simple algebras 
(for notations refer to List A-C).

\smallskip
\centerline{{\bf Table 1:} Building blocks $\g=\g_s\ltimes\rr$, $\g_s$ simple algebra and $\rr$ simple in $\cS_1$}

$\quad$\begin{tabular}{|c|c|}
 \hline &\\ [-0.01em] 
 $\g=\g_s\ltimes\rr$, $\qquad$  $\g_s\simeq \ssl(2),\,  \ssp(2n),\, \so^1(4n),\, \so^2(2m-1)$ & $\xymatrix{\bullet\ar@{->}@(ul,ur)}$\\ 
\hline &\\ [-0.01em] 
 $\g=\ssl(2n)\ltimes ad$, $\qquad$  $\g=\so^2(2m)\ltimes \Lambda^r(V)$,  $\,r$ is even
& $\xymatrix{\bullet\ar@{->}@(ul,ur) & \bullet\ar@{->}@(ul,ur)}$\\
\hline & \\ [-0.6em]
 $\g=\mathfrak{so}^2(2m)\ltimes\Lambda^r(V)$, $\,r$ is odd $\begin{array}{c} \ \\ \end{array}$ & 
 $\xymatrix{\bullet\ar@/^0.4pc/[rr]& &
\bullet\ar@/^0.4pc/[ll]}$\\ \hline
& \\ [-0.5em]
 $\g=\mathfrak{so}^2(4m)\ltimes\Lambda^{\pm}$  $\begin{array}{c} \ \\ \ \end{array}$ & 
 $\xymatrix{\bullet\ar@{->}@(ul,ur) &  \bullet}$\\ 
\hline &\\
[-0.5em]
 $\g=\g_s\ltimes\rr$,  $\ \g_s\simeq \mathfrak{sl}(2n),\,\mathfrak{so}^2(4m+2)$, $\,\rr^*\neq \rr$
 & 
 $\xymatrix{\bullet\ar[rr]& &
\bullet}$\\ \hline
\end{tabular}

\bigskip

\centerline{{\bf Table 2:} Building blocks 
 $\g=(\g_1\oplus\g_2)\ltimes S_1\otimes S_2$, $\g_i$ simple, $S_i$ simple module in $\cS_{\half}$.}

\begin{tabular}{|c|c|}
 \hline
 & \\
[-0.5em]
 $\g=(\g_1\oplus \g_2)\ltimes S_1\otimes S_2$, $\quad$  $\g_i\simeq \ssl(2), \ssp(2n), \so^1(4n), \so^2(2m+1)$ & 
 $\xymatrix{S_1\ar@/^0.4pc/[rr]& &
S_2\ar@/^0.4pc/[ll]}$\\ \hline
& \\
[-0.3em]
$\g=(\g_1\oplus \g_2)\ltimes S_1\otimes S_2$,  $\ $ $\begin{array}{c} \g_1\simeq \ssl(2), \ssp(2n), \so^1(4n), \so^2(2m+1),\\  \g_2\simeq \ssl(2n), \so^2(4m+2) \\
\end{array}$ & 
 $\xymatrix{S_1\ar[rr]& &
S_2
\\ &&S^*_2\ar[llu]}$\\ \hline
& \\
[-0.5em]
 $\g=(\g_1\oplus \mathfrak{so}^2(4m))\ltimes S_1\otimes S_2$, $\ $  $\g_1\simeq \ssl(2), \ssp(2n), \so^1(4n), \so^2(2m+1)$  
 &
 $\xymatrix{S_1\ar@/^0.4pc/[rr]& &
S_2\ar@/^0.4pc/[ll]
\\ && \bullet}$\\ \hline
& \\
[-0.5em]
 $\g=(\g_1\oplus \g_2)\ltimes S_1\otimes S_2$, $\qquad$  $\g_i\simeq \ssl(2n),\, \so^2(4m+2)$ &
 $\xymatrix{S_1& &
S_2
\\ S_1^*\ar[urr] && S_2^*\ar[ull]}$\\ \hline
& \\
[-0.5em]
 $\g=(\mathfrak{so}^2(4m)\oplus \g_2)\ltimes S_1\otimes S_2$, $\quad$  $\g_2\simeq \ssl(2n),\, \so^2(4m+2)$ &
 $\xymatrix{S_1\ar[rr]& &
S_2
\\ \bullet &&S^*_2\ar[llu]}$\\ \hline
& \\
[-0.5em]
 $\g=(\mathfrak{so}^2(4m)\oplus \mathfrak{so}^2_{4k})\ltimes S_1\otimes S_2$   &
 $\xymatrix{S_1 \ar@/^0.4pc/[rr]& &
S_2\ar@/^0.4pc/[ll]
\\ \bullet&& \bullet}$\\ \hline
\end{tabular}

\begin{Rem}
All quivers in Table 1 and Table 2 have all relations zero, except for the following algebras

\noindent $\begin{array}{l}
 (\ssl(2)\oplus \ssp(2n))\ltimes L\otimes V,\\
 (\ssl(2)\oplus \so^2(8m+1))\ltimes L\otimes \Gamma,\\ 
 (\ssl(2)\oplus \so^2(8m+7))\ltimes L\otimes \Gamma,
 \end{array}$  $\qquad$
 with the quiver $Q_1$ $\qquad$ $\begin{array}{c} \xymatrix{\bullet\ar@/^0.4pc/[rr]^{\alpha} & &
\bullet\ar@/^0.4pc/[ll]^{\beta}}\\ \alpha\beta=0,\end{array}$

\noindent and $\,(\ssl(2)\oplus \so^2(8m))\ltimes L\otimes \Gamma^+$,  $(\ssl(2)\oplus \so^2(8m))\ltimes L\otimes \Gamma^-$ 
 with the quiver $\,Q_1\cup {\bullet}$. These algebras correspond to the building block $\mathcal B^1=\mathcal{A}_1\circ S(W)$,
 $\dim W=1$ of Proposition~\ref{mutiplicity}(a).
\end{Rem}
\bigskip

\subsection{Building blocks with non-zero relations}
In this part we describe quivers with relations corresponding to the building blocks from Proposition~\ref{mutiplicity} and 
Lemma~\ref{Clifford example}. To simplify notations for the relations instead of thick arrows we will assume $\dim W=2$  (and $\dim W'=1$ in case Proposition~\ref{mutiplicity} (c)).
\medskip

\noindent
\begin{tabular}{|c|c|}
 \hline &\\ [-0.5em]
$\begin{array}{l}
 (\ssl(2)\oplus \ssp(2n))\ltimes L\otimes V\otimes W,\\
 (\ssl(2)\oplus \so^2(8m+1))\ltimes L\otimes \Gamma\otimes W,\\ 
 (\ssl(2)\oplus \so^2(8m+7))\ltimes L\otimes \Gamma\otimes W,\\
 (\ssl(2)\oplus \so^2(8m))\ltimes L\otimes \Gamma^{\pm}\otimes W,\\
 \end{array}$ 
 & $\begin{array}{c} \mathcal{A}_1\circ S(W): \qquad  \xymatrix{L\ar@/^0.4pc/[rr]^{\alpha} \ar@{-->}@/^1.3pc/[rr]^{\gamma}& &
\bullet\ar@/^0.4pc/[ll]^{\beta}\ar@{-->}@/^1.3pc/[ll]^{\delta}}\\
\delta\alpha=\beta\gamma,\ \alpha\beta=\gamma\delta=\gamma\beta=\alpha\delta=0\end{array}$\\
 \hline &\\ [-0.5em]
$\begin{array}{l}
  (\ssl(2)\oplus \so^1(4n))\ltimes L\otimes V\otimes W,\\
 (\ssl(2)\oplus \so^2(8m+3))\ltimes L\otimes \Gamma\otimes W,\\ 
 (\ssl(2)\oplus \so^2(8m+5))\ltimes L\otimes \Gamma\otimes W,\\
 (\ssl(2)\oplus \so^2(8m+4))\ltimes L\otimes \Gamma^{\pm}\otimes W,\\
 \end{array}$ 
 & $\begin{array}{c} \mathcal{A}_1\circ \Lambda(W): \qquad  \xymatrix{L\ar@/^0.4pc/[rr]^{\alpha} \ar@{-->}@/^1.3pc/[rr]^{\gamma}& &
\bullet\ar@/^0.4pc/[ll]^{\beta}\ar@{-->}@/^1.3pc/[ll]^{\delta}}\\
\delta\alpha=-\beta\gamma,\ \beta\alpha=\delta\gamma=\alpha\beta=\gamma\delta=\gamma\beta=\alpha\delta=0\end{array}$\\
 \hline &\\ [-0.5em]
 $\begin{array}{l}
  (\ssl(2)\oplus \ssl(2n))\ltimes \qquad\\
  \qquad \qquad \quad  (L\otimes V\otimes W\oplus L\otimes V^*\otimes W'),\\
(\ssl(2)\oplus \so^2(4m+2))\ltimes \qquad \\
\qquad \qquad \quad (L\otimes \Gamma^+\otimes W\oplus L\otimes \Gamma^-\otimes W'),
 \end{array}$ 
 & $\begin{array}{c} \mathcal{A}_2\circ S(W\oplus W')\\ 
  \xymatrix{L\ar@/^0.4pc/[rr]^{\alpha} \ar@{-->}@/^1.3pc/[rr]^{\xi}& &
\bullet\ar@/^0.4pc/[ll]^{\beta}\ar@{-->}@/^1.3pc/[ll]^{\eta}\ar@/^0.4pc/[rr]^{\delta} && \bullet\ar@/^0.4pc/[ll]^{\gamma} }\\
\beta\alpha=\delta\gamma=\beta\gamma=\delta\alpha=\delta\xi=\beta\xi=\eta\alpha=\eta\gamma=0,\\ 
\xi\eta=\alpha\beta=\gamma\delta\end{array}$\\
\hline
 &\\ [-0.5em]
$\begin{array}{l}
 \\
 \g=\mathfrak{so}^2(2m+1)\ltimes V\otimes W\\
 \\
 \end{array}$ & $\xymatrix{\bullet\ar@{->}@(ul,ur)^{\alpha}\ar@{-->}@(dl,dr)_{\beta}}$ 
  $\qquad$ $\alpha^2=\beta^2=0$,  $\alpha\beta=\beta\alpha$\\
\hline
 &\\ [-0.5em]
  $\begin{array}{l}\g=\mathfrak{so}^2(2m)\ltimes V\otimes W\\
\g=(\mathfrak{sl}(2)\oplus\mathfrak{sl}(2))\ltimes L\otimes L'\otimes W\\
  \end{array}$ & $\begin{array}{c} \xymatrix{\bullet\ar@/^0.4pc/[rr]^{\alpha} \ar@{-->}@/^1.3pc/[rr]^{\gamma}& &
\bullet\ar@/^0.4pc/[ll]^{\beta}\ar@{-->}@/^1.3pc/[ll]^{\delta}}\\  \beta\alpha=\delta\gamma,\ \alpha\beta=\gamma\delta,\ 
\beta\gamma=\gamma\beta=\delta\alpha=\gamma\delta=0\end{array}$\\
\hline

\end{tabular}

\bigskip

\section{Appendix}
In this section we collect necessary information about tensor product, bilinear forms on $\cS_{\frac12}$ and duality of modules 
in $\cS_{\frac12}$ and $\cS_1$.
\begin{lem}\label{lemma-dual-form}
(1) For $\g=\ssl(2)$ and $\g=\so^1(4n)$ standard module $V$ is self-dual and has an invariant skew-symmetric form.

\noindent (2) For $\g=\ssp(2n)$ standard module $V$ is self-dual and has an invariant symmetric form.

\noindent (3) For $\g=\so^2(2m+1)$ spinor module $\Gamma$ is self-dual. Moreover it has an invariant symmetric form when $m\equiv 0\,{\rm or}\, 3\mod 4$ and an invariant skew-symmetric form when 
$m\equiv 1\,{\rm or}\,2\mod 4$.

\noindent (4) For $\g=\so^2(4m)$ spinor modules $\Gamma^+$ and $\Gamma^-$ are both self-dual and have an
invariant symmetric form when $m$ is even and an invariant skew-symmetric form when $m$ is odd.

\noindent (5) For $\g=\so^2(4m+2)$ spinor module $\Gamma^+$ is dual to spinor module $\Gamma^-$. 

\noindent (6) For $\g=\ssl(2), \ssp(2n), \so^1(4n), \so^2(2m+1), \so^2(4m)$ all simple modules in $\cS_1$ are self-dual.

\noindent (7) For $\g=\ssl(2n)$, $ad$ is self-dual, while $(S^2V)^*=S^2(V^*)$ and $(\Lambda^2 V)^*=\Lambda^2(V^*)$.

\noindent (8) For $\g=\so^2(4m+2)$ all simple modules in $\cS_1$ are self-dual except $(\Lambda^{\pm})^*=\Lambda^{\mp}$.
\begin{proof} For proof use for example \cite{Fulton}, for spinor representation \cite{Deligne}.

\end{proof}
\end{lem}
\noindent Next we calculate $(M\otimes N)^{s}$ for simple $M\in  \cS_{\frac12}$, $N \in \mathcal S\cup \cS_{\frac12}$. 
Here $(\ )^s$ denotes a restriction to modules in $ \cS_{\frac12}\cup\{tr\}$, where $tr$
stays for trivial representation.
\begin{lem}\label{lemma-tensor} 
\noindent (1) For $\g=\mathfrak{sl}(2m)$,  $m\geq 3$, $U\in \{V, V^*\}$ 
$$
\begin{array}{llll}
(U\otimes\Lambda^2 U)^s=0, &(U\otimes S^2 U)^s=0, & (U^*\otimes\Lambda^2 U)^s=U, &(U^*\otimes S^2 U)^s=U,\\
(U\otimes ad)^s=U, & (U\otimes U)^s=0, &(U\otimes U^*)^s=tr. &\end{array}
$$
\noindent (2) For $\g=\mathfrak{sp}(2m)$, $m\geq 3$
$$
(V\otimes ad)^s=V,\qquad (V\otimes \overline{\Lambda^2 V})^s=V,\qquad (V\otimes V)^s=tr.
$$
\noindent (3) For $\g=\mathfrak{so}(4m)$, $m\geq 3$
$$
(V\otimes ad)^s=V,\qquad (V\otimes \overline{S^2 V})^s=V,\qquad (V\otimes V)^s=tr, \qquad (V\otimes\Gamma^+)^s=0.
$$
\noindent (4) For $\g=\so(2m+1)$,  $m\geq 2$ we have 
$
(\Gamma \otimes\Lambda^r V)^s=\Gamma$, $1\leq r\leq m$  and $(\Gamma\otimes \Gamma)^s=tr$.

\noindent (5) For $\g=\so(2m)$, $m\geq 3$, $1\leq r\leq m-1$
$$
(\Gamma^{\pm}\otimes\Lambda^r V)^s =\begin{cases}\Gamma^{\pm}, \  {\rm if\ } r {\rm\ is\ even},\\
\Gamma^{\mp}, \  {\rm if\ } r {\rm\ is\ odd,}\end{cases}\quad
(\Gamma^{\pm}\otimes \Gamma^{\pm} )^s =\begin{cases} tr, \  {\rm if\ } m {\rm\ is\ even},\\
0, \  {\rm if\ } m {\rm\ is\ odd},\end{cases}
$$
$$
(\Gamma^{\pm}\otimes\Lambda^{\pm})^s =\begin{cases}\Gamma^{\pm}, \  {\rm if\ } m {\rm\ is\ even},\\
0, \  {\rm if\ } m {\rm\ is\ odd},\end{cases}
\quad
(\Gamma^{+}\otimes\Gamma^{-})^s =\begin{cases}0, \  {\rm if\ } m {\rm\ is\ even},\\
tr, \  {\rm if\ } m {\rm\ is\ odd},\end{cases}
$$
$$
(\Gamma^{\pm}\otimes\Lambda^{\mp})^s =\begin{cases}0, \  {\rm if\ } m {\rm\ is\ even},\\
\Gamma^{\mp}, \  {\rm if\ } m {\rm\ is\ odd},\end{cases}.
$$
\end{lem}
\begin{proof} The formulas for tensor products are given in \cite{Vin}, table\ 5, applying $(\ )^{s}$ is straightforward. \end{proof}

\section{Acknowledgments}
The authors are grateful to V. Bekkert for helpful discussions. IK acknowledges financial support from the 
Guangdong Basic and Applied Basic Research Foundation (grant 2024A1515013079) and the NSFC (grant 12350710787),
VS was partially supported by the NSF grant  2019694.

\end{document}